\documentclass[conference,letterpaper]{IEEEtran}

\usepackage[utf8]{inputenc} 
\usepackage[T1]{fontenc}
\usepackage{url}
\usepackage{ifthen}
\usepackage{cite}
\usepackage[cmex10]{amsmath} 

\usepackage[utf8]{inputenc} 
\usepackage[T1]{fontenc}    
\usepackage{booktabs}       
\usepackage{amsfonts}       
\usepackage{nicefrac}       
\usepackage{amsmath} 
\usepackage{times}
\usepackage{subfig}
\usepackage{amsmath}
\usepackage{dsfont}
\usepackage{xr}
\usepackage{tikz}
\usepackage{pgfplots}
\usepackage{adjustbox}
\usepackage{bm, braket}
\usepackage{amsthm}
\usepackage{doi}
\usepackage{dsfont}
\usepackage{mathtools} 
\usepackage{nicematrix}
\usepackage{extarrows} 
\usepackage{hyperref}
\usepackage{balance}

\newtheorem{theorem}{Theorem}

\newtheorem{lemma}{Lemma}

\newtheorem{assumption}{Assumption}
\newtheorem{definition}{Definition}
\DeclareMathOperator*{\argmax}{arg\,max}

\newcommand{\bSigma}{\bm{\Sigma}}

\newcommand{\bmu}{\bm{\mu}}

\newcommand{\tbSigz}{\tilde{\bSigma}_\bZ}

\newcommand{\bA}{\mathbf{A}}
\newcommand{\bB}{\mathbf{B}}

\newcommand{\bI}{\mathbf{I}}

\newcommand{\bS}{\mathbf{S}}

\newcommand{\bU}{\mathbf{U}}
\newcommand{\bV}{\mathbf{V}}
\newcommand{\bW}{\mathbf{W}}
\newcommand{\bX}{\mathbf{X}}
\newcommand{\bY}{\mathbf{Y}}
\newcommand{\bZ}{\mathbf{Z}}

\newcommand{\bx}{\mathbf{x}}

\newcommand{\bz}{\mathbf{z}}

\newcommand{\bbE}{\mathbb{E}}
\newcommand{\bbR}{\mathbb{R}}
\newcommand{\bbP}{\mathbb{P}}

\newcommand{\bbQ}{\mathbb{Q}}

\newcommand{\tr}{\text{Tr}}
\newcommand{\diag}{\text{diag}}

\begin{document}
\title{Robust Mean Estimation With Auxiliary Samples} 

\author{%
 \IEEEauthorblockN{Barron Han$^*$, Danil Akhtiamov$^*$, Reza Ghane$^*$, Babak Hassibi}
\IEEEauthorblockA{Department of Electrical Engineering\\
                   Caltech\\
                   Pasadena, California\\
                   $^*$Equal contribution \\
                   Email: \{bshan, dakhtiam, rghanekh, hassibi\}@caltech.edu}}


\maketitle


\begin{abstract}

In data-driven learning and inference tasks, the high cost of acquiring samples from the target distribution often limits performance. A common strategy to mitigate this challenge is to augment the limited target samples with data from a more accessible "auxiliary" distribution. This paper establishes fundamental limits of this approach by analyzing the improvement in the mean square error (MSE) when estimating the mean of the target distribution. Using the Wasserstein-2 metric to quantify the distance between distributions, we derive expressions for the worst-case MSE when samples are drawn (with labels) from both a target distribution and an auxiliary distribution within a specified Wasserstein-2 distance from the target distribution. We explicitly characterize the achievable MSE and the optimal estimator in terms of the problem dimension, the number of samples from the target and auxiliary distributions, the Wasserstein-2 distance, and the covariance of the target distribution. We note that utilizing samples from the auxiliary distribution effectively improves the MSE when the squared radius of the Wasserstein-2 uncertainty ball is small compared to the variance of the true distribution and the number of samples from the true distribution is limited. Numerical simulations in the Gaussian location model illustrate the theoretical findings.
\end{abstract}

\section{Introduction and Motivation}
While modern data-driven techniques and deep learning models continue to achieve remarkable success in various inference tasks, access to training data remains a significant challenge. Often, acquiring data from the true (target) distribution is expensive or time-consuming. One technique is to augment data from the true distribution with data from auxiliary distributions that closely approximate the target distribution. For example, while collecting real-world data for self-driving vehicles is costly, simulated environments provide a more affordable alternative that can still be utilized for learning and inference. More explicitly, consider machine learning model training, where $\bW$ are the network weights, $(\bX, \bY)$ are the data and labels, respectively. We often approximate the true expected loss: $\mathbb{E}\left[l(\bW; \bX, \bY \right)]$ with respect to the underlying distribution of the data with the empirical loss $\frac{1}{n}\sum_{i=1}^n l(\bW; x_i,y_i)$. Auxiliary data $\bZ$ can be used to complement the empirical loss and provide a better estimate of the true expected loss. Another promising application involves employing the output of predictive or generative models to improve performance in inference tasks. For example, \cite{zrnicppi} studies statistical methods for inference when a small amount of true data is supplemented by a large number of AI generated predictions. Data from generative models have been shown to improve the generalization capabilities of classifiers such as in \cite{cycleGAN}.

This paper is motivated by the following questions: How closely must an auxiliary distribution approximate the target distribution for its samples to be useful in making inferences about the true distribution? How can we design optimal estimators for the statistics of the true distribution using samples from both the true and auxiliary distributions? Can we establish worst-case performance guarantees for statistical estimators when samples are drawn from an adversarially chosen auxiliary distribution? We aim to address such questions by analyzing the fundamental problem of mean estimation.

Robust estimation using auxiliary data has a rich history in prior literature. The ``data contamination" setting considers the auxiliary data as contaminated samples from the true distribution. The samples are not labeled so it is unknown in advance which data points are contaminated.
In his seminal work, Huber introduced the $\epsilon$-contamination model in which samples are drawn from a mixture distribution $(1-\epsilon) \bbP + \epsilon \bbQ$ where $\bbP$ is the true distribution for the data and $\bbQ$ represents a contaminated distribution. Moreover, the same article studies the case where $\bbP$ is normal, with $\bbQ$ chosen within a total variation (TV) ball centered at $\bbP$. Variations of Huber's $\epsilon$-contamination problem such as \cite{chen2016general} consider various metrics to measure the distance between the true distribution $\bbP$ and contaminated distribution $\bbQ$, including the TV, Kolmogorov, Levy, and Prokhov distances. 
We refer to \cite{loh2024theoretical} for an overview of robust statistics. Our problem setting differs from Huber's primarily since it is known a priori which samples belong to the true distribution and which belong to the auxiliary distribution. Consequently, the samples from the auxiliary distribution can only help inference -- we can always choose to ignore the auxiliary samples. Our proposed setting more closely models data augmentation problems in modern machine learning and inference.

Another setting from past work involves all samples originating from a shifted distribution, where the objective is to design statistical estimators that remain robust to such distribution shifts. Wasserstein distances often lead to tractable formulations for such problems, as they provide a natural measure of distance between the true and shifted distributions \cite{mohajerin2018data, aolaritei2023, kundro}. An advantage of the Wasserstein distance is that it imposes minimal assumptions on the underlying noise distribution. For instance, the $\mathcal{W}_2$ distance is symmetric and well defined even for distributions with nonoverlapping support, properties not shared by the KL divergence. Moreover, the Wasserstein distance plays a fundamental role in fields such as optimal transport and is used for modeling discrepancies between an empirical distribution, derived from observed samples, and the true distribution.

For example, \cite{zhu2022generalized}, \cite{liu2023robust}, and \cite{chao2023statistical} examine robust mean estimation in scenarios where samples are drawn from the worst-case shifted distribution within a Wasserstein neighborhood of the target distribution. Unlike the contamination model, this framework assumes that the samples are generated entirely from the shifted distribution. \cite{zhu2022generalized} constructs estimators for this problem by projecting the empirical distribution onto a well-behaved family of distributions that contains the true distribution. For Wasserstein or TV distribution shifts, \cite{zhu2022generalized} shows that this method asymptotically achieves the optimal error. \cite{liu2023robust} considers robust mean estimation, covariance estimation, and linear regression under distribution shifts. For mean estimation, which is the subject of this paper, they propose an estimator motivated by generative adverserial networks, and demonstrate that when the true distribution is Gaussian or elliptical, the estimator can achieve the optimal min-max $L_2$ estimation error. \cite{chao2023statistical} extend the prior framework by considering distribution shifts that can be independent or jointly shifted across data points. For the mean estimation problem with mean square loss, \cite{chao2023statistical} determined the exact min-max risk, the worst-case Wasserstein shift, and the least squares estimator.

The remainder of this paper is organized as follows. In Section \ref{sec:form}, we formalize the auxiliary data model, which produces  $n$ samples from the true distribution and $N$ samples from an auxiliary distribution, and the distributionally robust optimization objective for mean estimation. In Section \ref{sec:main}, we present our main results. We reduce the maximization for the worst-case MSE over all true and auxiliary distributions to an optimization over only their first and second moment statistics. By reducing the problem to Gaussians without loss of generality, we determine the optimal robust estimator, worst-case auxiliary distribution, and the min-max mean estimation error. We propose a few variations to the objective and present the optimal estimators for each case separately. In Section \ref{sec:proof} we present the proofs. Finally, in Section \ref{sec:sim} we demonstrate our theoretical findings with numerical simulations for the performance of the proposed estimator in a Gaussian location setting.

\section{Problem Formulation} \label{sec:form}
We start with the definition of the Wasserstein-$p$ ($\mathcal{W}_p$) distance.
\begin{definition}
    For two probability measures $\bbP_1$ and $\bbP_2$ on $\bbR^{d}$ with finite $p$th moment, the Wasserstein distance is defined as
    \begin{align} \label{eq:Wp}
        W_p(\bbP_1, \bbP_2) = \inf_{\gamma\in\Pi(\bbP_1, \bbP_2)} \Bigl(\int_{\bbR^d} \|\bx_1 - \bx_2\|_p^p \gamma(d\bx_1, d\bx_2)\Bigr)^{\frac{1}{p}}
    \end{align}
    Where $\Pi$ is the set of all couplings on $\bbP_1 \times \bbP_2$ such that for each $\gamma \in \Pi$, $\pi_{i\#} \gamma = \bbP_i$.
\end{definition}
In this work, we focus on the $\mathcal{W}_2$ distance, where $p = 2$ in \eqref{eq:Wp}. Consider two datasets, $\{\bX_i \in \mathbb{R}^d\}_{i=1}^n$ and $\{\bZ_j \in \mathbb{R}^d\}_{j=1}^N$, where $\{\bX_i\}_{i=1}^n$ are sampled i.i.d. from an unknown distribution $\bbP$, and $\{\bZ_j\}_{j=1}^N$ are sampled i.i.d. from another unknown distribution $\bbQ$. We know in advance that the distributions satisfy $\mathcal{W}_2(\bbP, \bbQ) \leq \epsilon$ for some $\epsilon > 0$. In practice, $\bbP$ represents the distribution to be learned, while $\bbQ$ serves as a surrogate for $\bbP$. The natural question arises: What is the optimal way to estimate the mean of $\bbP$, denoted by $\bmu_X$, using an estimator trained on both $\{\bX_i\}_{i=1}^n$ and $\{\bZ_j\}_{j=1}^N$, given $\mathcal{W}_2(\bbP, \bbQ) \leq \epsilon$? This question has motivated work such as \cite{zrnicppi}, but their setting assumes that the true data come in pairs $(\bX, \bY)$ where $\bX$ are features and $\bY$ are labels, and the goal is to utilize additional features $\bar \bX$ and the predictions, $\bar \bY$, of some unknown predictor as auxiliary samples to estimate with confidence some statistic on the underlying distribution of $\bY$.

Furthermore, we assume that $N \gg 1$ capturing the case when sampling from $\bZ$ is much cheaper than from $\bX$ and that the Frobenius norm of the covariance matrix of $\bX \sim \bbP$, denoted by $\bSigma_\bX$, is bounded below by $\delta^2$, which effectively imposes a lower bound on the uncertainty in $\bX$. Under these assumptions, our goal is to design an estimator $f$ that is robust to the choice of the true distribution $\bbP$ and also any auxiliary distribution $\bbQ$ in the $\epsilon$-neighborhood of $\bbP$. The optimal $f$ achieves the min-max rates:

\begin{equation} \label{eq:main}
    R^* = \min_{f} \max_{\mathcal{W}_2(\bbP,\bbQ) \le \epsilon, \|\bSigma_\bX\|_F \ge \delta^2}  \frac{\bbE_{\bbP,\bbQ} \|f(\bX,\bZ)- \bmu_X\|^2}{\|\bSigma_\bX\|_F}.
\end{equation}

A normalization factor such as $|\bSigma_\bX|_F$ is necessary. Without it, for example, setting $\bSigma_\bX = \bSigma_\bZ = c\bI_d$ with $c \to \infty$ causes the objective to diverge to infinity. This requirement also motivates the constraint $|\bSigma_\bX|_F \geq \delta^2$, which ensures that the denominator remains bounded away from zero. Alternative normalization factors, such as $\tr(\bSigma_\bX)$ or $|\bSigma_\bX|_{op}$, are also viable and considered in Section \ref{sec:main}. Practically, the objective represents a ratio (up to a constant factor) between the MSE of the estimator $f$, and the error of an estimator that utilizes samples from the true distribution. For instance, the sample mean from the true distribution achieves a MSE of $\frac{1}{n} \tr(\bSigma_\bX)$. We make the following linearity assumption on the estimator:

\begin{assumption} \label{lem: linear}
    (Linear estimator structure) For the objectives in \eqref{eq:main}, we assume the estimator, $f$, is linear with respect to the data $\bX, \bZ$ and of the form: \[f(\bX,\bZ) = \bA\frac{\sum \bX_i}{n} + \bB \frac{\sum \bZ_j}{N}\]
    where $\bA, \bB$ are constant matrices to determine.
\end{assumption}

\section{Main results} \label{sec:main}


    
    

Using the structure of the linear estimator from Assumption \ref{lem: linear}, \eqref{eq:main} can be reformulated as \eqref{eq: gen_mean_estim}. Similarly, we consider variations of the objective by normalizing using the trace and operator norms in \eqref{eq: mean_estim_trace} and \eqref{eq: mean_estim_op}.

\begin{subequations}
\label{eq:main_linear}
\begin{align}
    R_F^* &= \min_{\bA,\bB} \max_{\substack{\mathcal{W}_2(\bbP,\bbQ) \le \epsilon, \\ \|\bSigma_\bX\|_F \ge \delta^2}} \frac{\bbE_{\bbP,\bbQ} \|\bA\frac{\sum \bX_i}{n} + \bB \frac{\sum \bZ_j}{N} - \bmu_X\|^2}{\|\bSigma_\bX\|_F},
\label{eq: gen_mean_estim} \\
    R_{\mathrm{Tr}}^* &= \min_{\bA,\bB} \max_{\substack{\mathcal{W}_2(\bbP,\bbQ) \le \epsilon, \\ \mathrm{Tr}{(\bSigma_\bX)} \ge \delta^2}} \frac{\bbE_{\bbP,\bbQ} \|\bA\frac{\sum \bX_i}{n} + \bB \frac{\sum \bZ_j}{N} - \bmu_X\|^2}{\mathrm{Tr}{(\bSigma_\bX)}},
\label{eq: mean_estim_trace} \\
    R_{op}^* &= \min_{\bA,\bB} \max_{\substack{\mathcal{W}_2(\bbP,\bbQ) \le \epsilon, \\ |\bSigma_\bX|_{op} \ge \delta^2}} \frac{\bbE_{\bbP,\bbQ} \|\bA\frac{\sum \bX_i}{n} + \bB \frac{\sum \bZ_j}{N} - \bmu_X\|^2} {|\bSigma_\bX|_{op}}.
\label{eq: mean_estim_op}
\end{align}
\end{subequations}

First, we establish a lemma claiming that \eqref{eq:main_linear} is sensitive only to the first and second moment statistics of the true distribution $\bbP$ and the auxiliary distribution $\bbQ$.
\begin{lemma}
    (Reduction to Gaussians) Consider the optimizations in \eqref{eq:main_linear}. Without loss of generality, we can take \[\bbP = \mathcal{N}(\bmu_\bX, \bSigma_\bX), \quad \bbQ = \mathcal{N}(\bmu_\bZ, \bSigma_\bZ).\]
\end{lemma}

\begin{proof}
    The objective \eqref{eq: gen_mean_estim} can be expanded as:


\begin{align}\label{eq: obj_PQ}
& \min_{\bA, \bB} \max_{\mathcal{W}_2(\bbP,\bbQ) \le \epsilon, \|\bSigma_\bX\|_F \ge \delta^2} \nonumber \\
& \frac{\frac{1}{n}\tr{(\bA^T\bA\bSigma_\bX)} + \frac{1}{N}\tr{(\bB^T\bB\bSigma_\bZ)} + \|\bA \bmu_X + \bB \bmu_\bZ - \bmu_\bX\|^2}{\|\bSigma_\bX\|_F},
\end{align}
where $\bmu_Z$ and $\bSigma_Z$ denote the mean and covariance matrix of $\bZ \sim \bbQ$ respectively. 

Denote $$\bar{\bbP} = \mathcal{N}(\bmu_\bX, \bSigma_\bX) ~~~\mbox{and}~~~ \bar{\bbQ} = \mathcal{N}(\bmu_\bZ, \bSigma_\bZ)$$
Note that, if $\bbP$ and $\bbQ$ belong to the feasibility set of \eqref{eq: obj_PQ}, then so do $\bar{\bbP}$ and $\bar{\bbQ}$ because $\mathcal{W}_2(\bar{\bbP}, \bar{\bbQ}) \le \mathcal{W}_2(\bbP,\bbQ)$ according to Theorem 2.1 in \cite{Gelbrich1990}. Note, moreover, that since the objective of \eqref{eq: obj_PQ} depends only on the first and second order statistics of $\bbP$ and $\bbQ$, the objective does not change if we replace $\bbP$ and $\bbQ$ by $\bar{\bbP}$ and $\bar{\bbQ}$ respectively. Hence, we can take $\bbP = \mathcal{N}(\bmu_\bX, \bSigma_\bX)$ and $\bbQ = \mathcal{N}(\bmu_\bZ, \bSigma_\bZ)$ without loss of generality, which in turn implies: 
\begin{equation}\label{eq: W_2}
\mathcal{W}_2(\bbP, \bbQ)^2 = \|\bmu_\bX - \bmu_\bZ\|^2 + \mathrm{Tr} \left(\bSigma_\bX + \bSigma_\bZ - 2 (\bSigma_\bX^{\frac{1}{2}}\bSigma_\bZ\bSigma_\bX^{\frac{1}{2}})^{\frac{1}{2}}\right)
\end{equation}
\end{proof}

We arrive at our main theorems, which presents the optimal estimators for \eqref{eq:main_linear} and the optimal values of the objectives.

\begin{theorem}\label{thm: main}
    (Optimal estimator for \eqref{eq: gen_mean_estim}) The optimal estimator is given by
    \[\bA = s\bI, \ \bB = \bI - \bA,\]
    where 
    \begin{equation} \label{eq: opt_weight}
        s = \frac{\frac{\sqrt{d}}{N} + \frac{\epsilon^2}{\delta^2}}{\frac{\sqrt{d}}{N} + \frac{\sqrt{d}}{n} + \frac{\epsilon^2}{\delta^2}}.
    \end{equation}
    The optimal value of the objective:
    \begin{equation}
    R_F^* = \frac{\sqrt{d}(\frac{\delta^2}{\epsilon^2} \sqrt{d} + N)}{\frac{\delta^2}{\epsilon^2} \sqrt{d} (n+N) + \epsilon^2 n N},
    \end{equation}
    where $n$ is the number of samples from the true distribution, $N$ is the number of samples from the auxiliary distribution, $d$ is the problem dimension, and $\epsilon$ is the $\mathcal{W}_2$ uncertainty radius.
\end{theorem}

\begin{theorem}\label{cor: trace}
(Optimal estimator for \eqref{eq: mean_estim_trace}) The optimal estimator is given by
\begin{equation}
    \bA = s\bI, \ \bB = \bI - \bA
\end{equation}
where
\begin{equation}
s = \frac{\frac{\epsilon^2}{\delta^2} + \frac{1}{N}}{\frac{\epsilon^2}{\delta^2} + \frac{1}{n}+ \frac{1}{N}}, \quad \text{and} \quad R_{\mathrm{Tr}}^* = \frac{\frac{\delta^2}{\epsilon^2}+N}{\frac{\delta^2}{\epsilon^2}(n+N) + nN},
\end{equation}
with $n, N, \delta, \epsilon$ defined in Theorem \ref{thm: main}.
\end{theorem}

\begin{theorem}\label{cor: op}
(Optimal estimator for \ref{eq: mean_estim_op}) The optimal estimator is given by

\begin{equation}
    \bA = s\bI, \ \bB = \bI - \bA
\end{equation}
where
\begin{equation}
s = \frac{\frac{\epsilon^2}{\delta^2} + \frac{d}{N}}{\frac{\epsilon^2}{\delta^2} + \frac{d}{n}+ \frac{d}{N}}, \quad \text{and} \quad R_{op}^* = \frac{d (\frac{\delta^2}{\epsilon^2}d+ N)}{\frac{\delta^2}{\epsilon^2} d(n+N) + n N},
\end{equation}
with $n, N, \delta, \epsilon$ defined in Theorem \ref{thm: main}.
\end{theorem}

The key insight from Theorems \ref{thm: main}, \ref{cor: trace} and \ref{cor: op} is that normalizing the objective using various matrix norms on $\bSigma_\bX$ does not alter the structure of the optimal estimator or the interactions between $n$, $N$, $\delta$, and $\epsilon$ in the optimal weighting constant, $s$, or the optimal objective value, $R^*$, except for introducing a different dependence on the dimension of the problem $d$.

In general, the weighting factor approaches $s \to 1$, implying the estimate becomes more sensitive to the true samples while disregarding the auxiliary samples, if $\epsilon^2 \gg \frac{\delta^2}{n}$. This condition suggests that the auxiliary distribution provides a poor approximation of the true distribution, and the mean estimation task is relatively easy using only the true samples. 

Let $R^*$ be the optimal rate in Theorem \ref{cor: trace}. The expression
\begin{equation}
    n R^* = \frac{\frac{\delta^2}{\epsilon^2}+N}{\frac{\delta^2}{\epsilon^2}(1+\frac{N}{n}) + N}
\end{equation}
is the ratio between the worst-case MSE of the robust estimator and the MSE of the sample mean from the true distribution. This ratio is small when the uncertainty in the true distribution (parametrized by $\delta^2$) is large compared to the square of the $\mathcal{W}_2$ radius $\epsilon$. In fact, $\frac{\delta^2}{\epsilon^2}$ should be on the order of $N$, the number of auxiliary samples. Further, $\frac{N}{n}$ should be large, indicating that the number of true samples is relatively limited.

\section{Proof of Theorem \ref{thm: main}} \label{sec:proof}
We provide a detailed proof of Theorem \ref{thm: main}. The proofs for Theorems \ref{cor: trace} and \ref{cor: op} require slight algebraic modifications and are excluded for brevity.


\begin{proof}
Note that \eqref{eq: obj_PQ} is unbounded unless $\bB = \bI - \bA$. Indeed, fix arbitrary $\bA,\ \bB, \ \bSigma_\bX,$ and $\bSigma_\bZ$ and take $\bmu_\bX = \bmu_\bZ$ align with the principal singular direction of $\bA + \bB  - \bI$ satisfying $\|\bmu_\bX\| \to \infty$ and such that \eqref{eq: W_2} is upper-bounded by $\epsilon$. We  have
\begin{equation}
   \|\bA \bmu_\bX + \bB \bmu_\bZ - \bmu_\bX\|^2 = \|(\bA + \bB - \bI)\bmu_\bX\|^2 \to \infty \\
\end{equation}
unless $\bA + \bB - \bI = 0$.

From the above remark, \eqref{eq: obj_PQ} takes the following form: 

\begin{align}\label{eq: obj_A}
& \min_{\bA} \max_{\mathcal{W}_2(\bbP,\bbQ) \le \epsilon, \|\bSigma_\bX\|_F \ge \delta^2}  \frac{\|(\bA - \bI)(\bmu_\bX - \bmu_\bZ)\|^2}{\|\bSigma_\bX\|_F} + \nonumber \\
& \frac{\frac{1}{n}\tr{(\bA^T\bA\bSigma_\bX)} +\frac{1}{N}\tr{((\bI - \bA )^T(\bI - \bA )\bSigma_\bZ)}}{\|\bSigma_\bX\|_F}.
\end{align}

Since $\|(\bA - \bI)(\bmu_\bX - \bmu_\bZ)\|^2$ is the only term in the objective that depends on the direction of $\bmu_\bX - \bmu_\bZ$ and $\mathcal{W}_2(\bbP, \bbQ)$ depends only on the norm of $\bmu_\bX - \bmu_\bZ$ as well, it is clear that $\bmu_\bX - \bmu_\bZ$ will align with the principal singular direction of $\bA$ and we obtain the following:

\begin{align}
& \min_{\bA} \max_{\mathcal{W}_2(\bbP,\bbQ) \le \epsilon, \|\bSigma_\bX\|_F \ge \delta^2}  \frac{\sigma_{\max}(\bA - \bI)^2\|\bmu_\bX - \bmu_\bZ\|^2}{\|\bSigma_\bX\|_F} + \nonumber \\
& \frac{\frac{1}{n}\tr{(\bA^T\bA\bSigma_\bX)} +\frac{1}{N}\tr{((\bI - \bA )^T(\bI - \bA )\bSigma_\bZ)}}{\|\bSigma_\bX\|_F}.
\end{align}

Hence, the constraint $\mathcal{W}_2(\bbP, \bbQ) = \epsilon$ is active and we arrive at:
\begin{align}\label{obj_sigma_max}
& \min_{\bA} \max_{\mathcal{W}_2(\bbP,\bbQ) \le \epsilon, \|\bSigma_\bX\|_F \ge \delta^2}  \nonumber  \\
&\frac{\sigma_{\max}(\bA - \bI)^2(\epsilon^2-\tr{(\bSigma_\bX + \bSigma_\bZ - 2 (\bSigma_\bX^{\frac{1}{2}}\bSigma_\bZ}\bSigma_\bX^{\frac{1}{2}})^{\frac{1}{2}}))}{\|\bSigma_\bX\|_F} + \nonumber \\
& \frac{\frac{1}{n}\tr{(\bA^T\bA\bSigma_\bX)} +\frac{1}{N}\tr{((\bI - \bA )^T(\bI - \bA )\bSigma_\bZ)}}{\|\bSigma_\bX\|_F}
\end{align}

Define $\tbSigz = (\bSigma_\bX^{\frac{1}{2}}\bSigma_\bZ\bSigma_\bX^{\frac{1}{2}})^{\frac{1}{2}}$, then $\bSigma_\bZ = \bSigma_\bX^{-\frac{1}{2}}\tbSigz^2\bSigma_\bX^{-\frac{1}{2}}$ and \eqref{obj_sigma_max} becomes,

\begin{equation}
    \begin{aligned}\label{obj_tSigz}
    & \min_{\bA} \max_{\mathcal{W}_2(\bbP,\bbQ) \le \epsilon, \|\bSigma_\bX\|_F \ge \delta^2}  \nonumber  \\
    &\frac{\sigma_{\max}(\bA - \bI)^2(\epsilon^2-\tr{(\bSigma_\bX + \bSigma_\bX^{-\frac{1}{2}}\tbSigz^2\bSigma_\bX^{-\frac{1}{2}} - 2 \tbSigz}))}{\|\bSigma_\bX\|_F} + \nonumber \\
    & \frac{\frac{1}{n}\tr{(\bA^T\bA\bSigma_\bX)} +\frac{1}{N}\tr{((\bI - \bA )^T(\bI - \bA )\bSigma_\bX^{-\frac{1}{2}}\tbSigz^2\bSigma_\bX^{-\frac{1}{2}})}}{\|\bSigma_\bX\|_F}
    \end{aligned}
\end{equation}

Let us take the matrix derivative by $\tbSigz$ and equate it to zero. Note that it does not take the constraint $\tbSigz \succ 0$ into account but as we will see the solution of the corresponding KKT conditions will turn out to be positive definite by itself:
\begin{equation}
    \begin{aligned}
    & \frac{2}{N}\bSigma_\bX^{-\frac{1}{2}}(\bI - \bA )^T(\bI - \bA )\bSigma_\bX^{-\frac{1}{2}}\tbSigz \\ 
    &- \sigma_{max}(\bA - \bI)^2(2\bSigma_\bX^{-1}\tbSigz - 2\bI)  = 0.
    \end{aligned}
\end{equation}

We recover $\tbSigz$ as
\begin{equation}
    \begin{aligned}
    & \tbSigz  = \sigma_{max}(\bA - \bI)^2 \cdot \Big(\bSigma_\bX^{-\frac{1}{2}}\big(\sigma_{max}(\bA - \bI)^2 \\
    &  - \frac{1}{N}(\bI - \bA )^T(\bI - \bA )\big)\bSigma_\bX^{-\frac{1}{2}}\Big)^{-1}
    \end{aligned}
\end{equation}

 We deduce that as $N \to \infty$, $\tbSigz $ uniformly converges to $\bSigma_\bX$ in the operator norm. Hence, for large enough $N$, $\tbSigz \approx \bSigma_\bX$ implying $\bSigma_\bZ \approx \bSigma_\bX$ and we obtain:
 

\begin{equation} \label{eq: before_holder}
\begin{aligned}
    \min_{\bA} \max_{\mathcal{W}_2(\bbP,\bbQ) \le \epsilon, \|\bSigma_\bX\|_F \ge \delta^2} & \frac{\sigma_{\max}(\bA - \bI)^2\epsilon^2 + \frac{1}{n}\tr{\big(\bA^T\bA\bSigma_\bX \big)}}{\|\bSigma_\bX\|_F} \\
    & + \frac{\frac{1}{N}\tr{((\bI - \bA )^T(\bI - \bA )\bSigma_\bX)}}{\|\bSigma_\bX\|_F}.
    \end{aligned}
\end{equation}

Hence, $\bSigma_\bX$ aligns with $\frac{\bA^T\bA}{n} + \frac{(\bI - \bA )^T(\bI - \bA )}{N}$, $\|\bSigma_\bX\|_F = \delta^2$ and

\begin{align}\label{eq: obj_A2}
& \min_{\bA} \frac{\sigma_{\max}(\bA - \bI)^2\epsilon^2}{\delta^2} + \|\frac{\bA^T\bA}{n} + \frac{(\bI - \bA )^T(\bI - \bA )}{N}\|_F.
\end{align}

Use the singular value decomposition $\bA = \bU \bS \bV$, where $\bU, \bV \in \mathbb{R}^d$ are unitary and $\bS = \diag(s_1,\dots,s_d)$ where $s_1 \ge s_2 \dots \ge s_d \ge 0$. The objective \eqref{eq: obj_A2} decomposes as 
\begin{align}\label{eq: obj_s_i}
& \min_{s_1 \ge s_2 \ge \dots \ge s_d \ge 0} \nonumber \\
& \frac{\max\{(s_1-1)^2, (s_d-1)^2\}\epsilon^2 }{\delta^2} + \sqrt{\sum_i \left(\frac{s_i^2}{n} + \frac{(1 - s_i )^2}{N}\right)^2}
\end{align}

Now note that for $d>i>1$ the optimization in $s_i$ is quadratic with a single constraint $s_1 \ge s_i \ge s_d$ and we have: 

$$s_i  = s := \begin{cases}
    & \frac{1}{1 + \frac{N}{n}} \text { if } s_1 \ge \frac{1}{1 + \frac{N}{n}}  \ge s_d \\
    & s_1 \text { if } s_1 < \frac{1}{1 + \frac{N}{n}}  \\
    & s_d \text { if } s_d > \frac{1}{1 + \frac{N}{n}}  \\
\end{cases}  $$  

Consider the cases above one by one:

$\bullet$ If $s_1 \ge \frac{1}{1 + \frac{N}{n}}  \ge s_d$, then $s_d = 1 + \frac{N}{n} = s $ as well because this choice of $s_d$ minimizes both $(1-s_d)^2$ and $\frac{s_d^2}{n} + \frac{(1 - s_d )^2}{N}$ under the given constraints and  $\max\{(s_1-1)^2, (s_d-1)^2\} = (s_d-1)^2$ because the optimal $s_d \le s_1 \le 1$.  Hence, \eqref{eq: obj_s_i} becomes:

\begin{align}
& \min_{s_1 \ge \frac{1}{1+\frac{N}{n}}} \frac{N^2\epsilon^2}{\delta^2(n+N)^2} + \sqrt{ \left(\frac{s_1^2}{n} + \frac{(1 - s_1 )^2}{N}\right)^2 + \frac{d-1}{(n+N)^2}}
\end{align}

Therefore, the optimal $s_1 = \frac{1}{1+\frac{N}{n}}$ as well, $\bA = \frac{1}{1+\frac{N}{n}}\bI$ and the value of the objective \eqref{eq: obj_s_i} equals 

\begin{align}\label{eq: first_obj}
\frac{N^2\epsilon^2}{(n+N)^2\delta^2} + \frac{\sqrt{d}}{n+N} = \frac{1}{n+N}\left(\frac{N^2\epsilon^2}{(n+N)\delta^2}+\sqrt{d}\right)
\end{align}

$\bullet$ If $s_d \ge \frac{1}{1+\frac{N}{n}}$, then as discussed earlier $s_2 = \dots = s_{d-1} =:s = s_d$, and \eqref{eq: obj_s_i} simplifies to 

\begin{align}
& \min_{s_1 \ge s \ge \frac{1}{1+\frac{N}{n}}} \frac{\max\{(s_1-1)^2, (s-1)^2\}\epsilon^2}{\delta^2} + \nonumber \\
& \sqrt{\left(\frac{s_1^2}{n} + \frac{(1 - s_1 )^2}{N}\right)^2 + (d-1)(\frac{s^2}{n} + \frac{(1 - s)^2}{N})^2}
\end{align}

Hence, $s_1 = s_d$ as well, because otherwise resetting $s_1 := s$ would decrease the first term and not change the second and we arrive at: 

\begin{align}\label{eq: second_case}
& \min_{s \ge \frac{1}{1+\frac{N}{n}}} \frac{(s-1)^2\epsilon^2}{\delta^2} + \sqrt{d}\left(\frac{s^2}{n} + \frac{(1 - s)^2}{N}\right)
\end{align}

$\bullet$ If $s_1 \le \frac{1}{1+\frac{N}{n}}$, similarly to the previous case we arrive at the following:

\begin{align}\label{eq: third_case}
& \min_{s \le \frac{1}{1+\frac{N}{n}}} \frac{(s-1)^2\epsilon^2}{\delta^2} + \sqrt{d}\left(\frac{s^2}{n} + \frac{(1 - s)^2}{N}\right)
\end{align}

Combining \eqref{eq: second_case} and \eqref{eq: third_case} altogether into one objective, we arrive at:

\begin{align}
& \min_{s} \frac{(s-1)^2\epsilon^2}{\delta^2} + \sqrt{d}\left(\frac{s^2}{n} + \frac{(1 - s)^2}{N}\right)
\end{align}

Hence, $s = \frac{\frac{\epsilon^2}{\delta^2}+\frac{\sqrt{d}}{N}}{\frac{\epsilon^2}{\delta^2}+\frac{\sqrt{d}}{N} + \frac{\sqrt{d}}{n}}$, $\bA = s\bI$ and the corresponding value of the objective is equal to 

\begin{align}\label{eq: second_obj}
    & \frac{(\frac{\epsilon^2}{\delta^2}+\frac{\sqrt{d}}{N})^2\frac{\sqrt{d}}{n} + (\frac{\epsilon^2}{\delta^2}+\frac{\sqrt{d}}{N})\frac{d}{n^2}}{(\frac{\epsilon^2}{\delta^2}+\frac{\sqrt{d}}{N} + \frac{\sqrt{d}}{n})^2} = \frac{(\frac{\epsilon^2}{\delta^2}+\frac{\sqrt{d}}{N})\frac{\sqrt{d}}{n}}{\frac{\epsilon^2}{\delta^2}+\frac{\sqrt{d}}{N} + \frac{\sqrt{d}}{n}} \nonumber \\
    & = \frac{(\frac{\epsilon^2N}{\delta^2}+\sqrt{d})\sqrt{d}}{\frac{\epsilon^2nN}{\delta^2}+\sqrt{d}(n+N)} = \frac{1}{n+N}\frac{(\frac{\epsilon^2N}{\delta^2}+\sqrt{d})\sqrt{d}}{\frac{\epsilon^2nN}{\delta^2(n+N)}+\sqrt{d}} 
\end{align}

\eqref{eq: second_obj} is upper-bounded by \eqref{eq: first_obj}. It suffices to show that 

\begin{align*}
    \frac{(\frac{\epsilon^2N}{\delta^2}+\sqrt{d})\sqrt{d}}{\frac{\epsilon^2nN}{\delta^2(n+N)}+\sqrt{d}} \le \frac{N^2\epsilon^2}{(n+N)\delta^2}+\sqrt{d}
\end{align*}



Simplifying the expression:
\begin{align*}
    \frac{\epsilon^2}{\delta^2}\frac{N^2}{n+N}\frac{\sqrt{d}}{\frac{\epsilon^2nN}{\delta^2(n+N)}+\sqrt{d}} \le \frac{N^2\epsilon^2}{(n+N)\delta^2},
\end{align*}
implying that \eqref{eq: second_obj} is the optimal solution to \eqref{eq: obj_A}. 
\end{proof}

\section{Numerical Experiments} \label{sec:sim}

To showcase our theoretical predictions, we conduct a numerical experiment as follows. We take $\bbP \sim \mathcal{N}(\mu_{\bbP}, \sigma\bI_d)$ and $\bbQ \sim \mathcal{N}(\mu_{\bbQ}, \sigma\bI_d)$ to be two $d$-dimensional  Gaussian distributions with different means and the same scalar covariance satisfying $\mathcal{W}_2(\bbP, \bbQ)^2 = \|\mu_{\bbP} - \mu_{\bbQ}\|^2 = \epsilon^2$ and $\|\bSigma_\bX\|_F = \delta^2$ and fix $\epsilon = \delta = 1$. We take the dimension $d = 200$ and generate $n = 20$ samples $\bx_1,\dots, \bx_n$ from $\bbP$ and $N = 1000$ samples $\bz_1,\dots,\bz_N$ from $\bbQ$. We then compare the average performances of three predictors for the "true" mean $\bbP$: averaging over the "true" data $\frac{\bx_1 + \dots + \bx_n}{n}$ plotted in green, averaging over all data $\frac{\bx_1 + \dots + \bx_n + \bz_1 + \dots + \bz_N}{n+N}$ plotted in blue, and the optimal weighted average from Theorem \ref{thm: main} plotted in red. As expected, the optimal estimator achieves lower error compared to both naive averaging of all the samples or averaging of only the samples from the true distribution. When the $\mathcal{W}_2$ is small, the optimal estimator essentially averages all the samples. As the $\mathcal{W}_2$ distance, $\epsilon$ grows large, the optimal error approaches the error of the sample average from the true distribution.

\begin{figure}
    \centering
    \includegraphics[width=8cm]{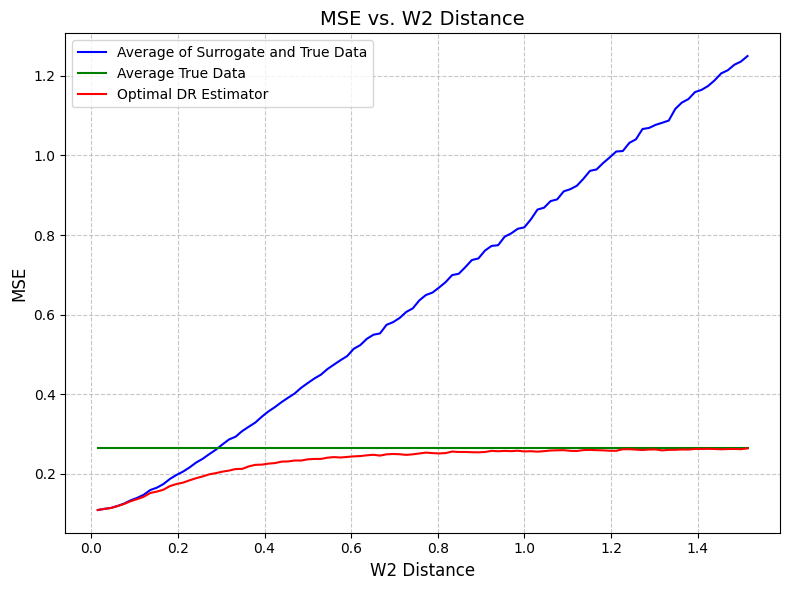}
    \caption{MSE of optimal estimator from Theorem \ref{thm: main}.}
    \label{fig:optest}
\end{figure}

\section{Conclusion}

We suggest a new framework for mean estimation using samples from both true and auxiliary distributions and find the worst-case optimal linear mean estimator. We show that incorporating auxiliary data for estimation always improves performance compared to just averaging over the true distribution. However, the improvement is marginal unless one of the following conditions holds: the auxiliary data distribution matches the original one closely, the original data is very noisy or there is a very limited number of samples from the original distribution.  

Potential directions for future work include extending the analyses to other statistical distances, such as $W_p$-distances for $p \ne 2$ and the KL-divergence, and exploring potential practical applications of the introduced framework.





\balance
\bibliographystyle{IEEEtran}
\bibliography{main}









\appendices

\section{Proof of Theorem \ref{cor: trace}}
All steps in the proof of Theorem \ref{thm: main} up to equation \eqref{eq: before_holder} remain unchanged modulo a different denominator and a different constraint on $\bSigma_\bX$, and we arrive at:

\begin{align}\label{eq: trace_ratio}
& \min_{\bA} \max_{\mathcal{W}_2(\bbP,\bbQ) \le \epsilon, \tr{(\bSigma_\bX)} \ge \delta^2}  \nonumber  \\
& \frac{\sigma_{\max}(\bA - \bI)^2\epsilon^2 + \frac{1}{n}\tr{(\bA^T\bA\bSigma_\bX)} +\frac{1}{N}\tr{((\bI - \bA )^T(\bI - \bA )\bSigma_\bX)}}{\tr{(\bSigma_\bX)}}
\end{align}

Note that at the optimal point $\tr{(\bSigma_\bX)} = \delta^2$ because the second term of \eqref{eq: trace_ratio}, namely the following expression, is invariant to rescalings of $\tr{(\bSigma_\bX)}$:
$$\frac{\frac{1}{n}\tr{(\bA^T\bA\bSigma_\bX)} +\frac{1}{N}\tr{((\bI - \bA )^T(\bI - \bA )\bSigma_\bX)}}{\tr{(\bSigma_\bX)}}$$

And the first term of \eqref{eq: trace_ratio}, namely the ratio below, increases when $\tr{(\bSigma_\bX)}$ decreases: 
\begin{align*}
\frac{\sigma_{\max}(\bA - \bI)^2\epsilon^2}{\tr{(\bSigma_\bX)}}
\end{align*}

Now, applying Holder's inequality, we obtain:

\begin{align*}
& \frac{\frac{1}{n}\tr{(\bA^T\bA\bSigma_\bX)} +\frac{1}{N}\tr{((\bI - \bA )^T(\bI - \bA )\bSigma_\bX)}}{\tr{(\bSigma_\bX)}} \\
& \le \|\frac{\bA^T\bA}{n}+\frac{(\bI - \bA )^T(\bI - \bA )\bSigma_\bX}{N}\|_{op}
\end{align*}

As such, the entire objective is upper-bounded by 

\begin{align*}
\frac{\sigma_{\max}(\bA - \bI)^2\epsilon^2}{\delta^2} + \|\frac{\bA^T\bA}{n}+\frac{(\bI - \bA )^T(\bI - \bA )}{N}\|_{op}
\end{align*}

Moreover, the equality is attained when $\bSigma_\bX$ is rank one and is supported on the principal direction of 
$\frac{\bA^T\bA}{n}+\frac{(\bI - \bA )^T(\bI - \bA )}{N}$. 

As such, we arrive at: 

\begin{align*}
\min_{\bA} \frac{\|\bA - \bI\|_{op}^2\epsilon^2}{\delta^2} + \|\frac{\bA^T\bA}{n}+\frac{(\bI - \bA )^T(\bI - \bA )}{N}\|_{op}
\end{align*}

Taking the SVD and diagonalizing the objective again similar to the proof of Theorem \ref{thm: main}, we arrive at

\begin{align*}
\min_{s_1 \ge \dots \ge s_d \ge 0} \max_{i=1,\dots,d} \frac{(s_i-1)^2\epsilon^2}{\delta^2} + \max_{i=1,\dots,d} \left(\frac{s_i^2}{n}+\frac{(1-s_i)^2}{N}\right)
\end{align*}

Hence, the minimum is achieved when $s_1 = \dots = s_d$ and the objective turns into: 

\begin{align*}
\min_{s_1 \ge \dots \ge s_d \ge 0} \max_{i=1,\dots,d} \frac{(s_i-1)^2\epsilon^2}{\delta^2} + \max_{i=1,\dots,d} \left(\frac{s_i^2}{n}+\frac{(1-s_i)^2}{N}\right)
\end{align*}

\begin{align*}
\min_{s \ge 0} \frac{(s-1)^2\epsilon^2}{\delta^2} + \left(\frac{s^2}{n}+\frac{(1-s)^2}{N}\right)
\end{align*}

Hence,

$$\bA = s\bI \text { and } s = \frac{\frac{\epsilon^2}{\delta^2} + \frac{1}{N}}{\frac{\epsilon^2}{\delta^2} + \frac{1}{n}+ \frac{1}{N}}$$

\section{Proof of Theorem \ref{cor: op}}

Similarly to the proof of Theorem \ref{cor: trace}, we arrive at:

\begin{align*}
\min_{\bA} \frac{\|\bA - \bI\|_{op}^2\epsilon^2}{\delta^2} + \tr{\left(\frac{\bA^T\bA}{n}+\frac{(\bI - \bA )^T(\bI - \bA )}{N}\right)}
\end{align*}

Taking the SVD:

\begin{align*}
\min_{s_1 \ge \dots \ge s_d \ge 0} \left(\max_{i=1,\dots,d} \frac{(s_i-1)^2\epsilon^2}{\delta^2}\right) + \sum_i \left(\frac{s_i^2}{n}+\frac{(1-s_i)^2}{N}\right)
\end{align*}

Denote $s_t = \argmax (s_i-1)^2$. WLOG assume $s_t \le 1$. Note that either $s_t \le \frac{1}{1 + \frac{N}{n}}$ and for any $i \ne t$ we have $s_i = \frac{1}{1 + \frac{N}{n}}$ , and then $s_t =  \frac{1}{1 + \frac{N}{n}}$ as well, or  $s_t > \frac{1}{1 + \frac{N}{n}}$ and $s_i = s_t$ for all $i$. In either case, all $s_i$ are equal to each other and denoting $s:=s_i$ we arrive at:

\begin{align*}
\min_{s > 0} \frac{(s-1)^2\epsilon^2}{\delta^2}+ d\left(\frac{s_i^2}{n}+\frac{(1-s_i)^2}{N}\right)
\end{align*}

Hence, $\bA = s\bI \text { where }  s = \frac{\frac{\epsilon^2}{\delta^2} + \frac{d}{N}}{\frac{\epsilon^2}{\delta^2} + \frac{d}{n} +\frac{d}{N}}$

\end{document}